\documentclass[10pt]{article}
\usepackage[english,activeacute]{babel}
\usepackage{amsmath,amsfonts,amssymb,amstext,amsthm,amscd,mathrsfs,amsbsy,xypic,centernot,graphicx,float,color}
\usepackage{hyperref} 
\usepackage{indentfirst}
\usepackage{xy}
\usepackage{fourier}

\newtheorem{teo}{Theorem}[section]
\newtheorem{pro}[teo]{Proposition}
\newtheorem{coro}[teo]{Corollary}
\newtheorem{lem}[teo]{Lemma}

\theoremstyle{definition}
\newtheorem{defi}[teo]{Definition}
\newtheorem{exam}[teo]{Example}
\newtheorem{rem}[teo]{Remark}

\newcommand{\N}{\mathbb N}
\newcommand{\Z}{\mathbb Z}

\newcommand{\R}{\mathbb R}
\newcommand{\C}{\mathbb C}

\newcommand{\Xo}{(X,0)}
\newcommand{\Oxo}{\mathcal{O}_{X,0}}

\newcommand{\Onxo}{\mathcal{O}_{\bar{X},0}}
\newcommand{\Oo}{\mathcal{O}}
\newcommand{\I}{\mathcal{I}}
\newcommand{\m}{\mathfrak{m}}

\newcommand{\A}{\mathcal{A}}
\newcommand{\B}{\mathcal{B}}
\newcommand{\Ga}{\Gamma}
\newcommand{\ga}{\gamma}

\newcommand{\sd}{\check{\sigma}}
\newcommand{\vp}{\varphi}
\newcommand{\ig}{I_{\Gamma}}
\newcommand{\xg}{X_{\Gamma}}
\DeclareMathOperator{\edim}{edim}
\DeclareMathOperator{\mult}{mult}

\newcommand{\pol}{\C[z_1,\ldots,z_n]}
\newcommand{\cs}{\C\{z_1,\ldots,z_n\}}
\newcommand{\fs}{\C[[z_1,\ldots,z_n]]}
\newcommand{\pg}{\C[\Ga]}
\newcommand{\pgg}{\C[t^{\ga_1},\ldots,t^{\ga_n}]}
\newcommand{\csg}{\C\{\Ga\}}
\newcommand{\fsg}{\C[[\Ga]]}

\begin{document}

\title{On the Lipschitz Saturation of Toric Singularities}

\author{ Daniel Duarte, 
Arturo E. Giles Flores}

\maketitle

\begin{abstract}
We begin the study of Lipschitz saturation for germs of toric singularities. By looking at their associated analytic algebras, we prove that if $(X,0)$ is a germ of toric singularity with smooth normalization then its Lipschitz saturation is again toric. Finally we show how to calculate the Lipschitz saturation for some families of toric singularities starting from the semigroup that defines them.
\end{abstract}

\section*{Introduction} 
  
For a germ $(X,0) \subset (\C^n,0)$ of reduced complex analytic singularity, the algebra of germs of Lipschitz meromorphic functions is an analytic algebra that sits between $O_{X,0}$ and its normalization $\overline{O_{X,0}}$. The definition is algebraic, and  was inspired by Zariski's theory of saturation, whose objective was to establish the foundations for an algebraic theory of equisingularity as seen in \cite{Z68}, \cite{Z71a}, \cite{Z71b} and \cite{Z75}.  It was first studied by Pham and Teissier in \cite{Ph-Te20}, where they prove that Lipschitz saturation and Zariski saturation coincide for hypersurfaces.  In this paper they also introduced a relative notion of Lipschitz saturation whose algebraic treatment was further studied in \cite{Bo74}, \cite{Bo75}, \cite{Lip75}. Beyond the algebraic relevance of the construction, the associated germ $(X^s,0)$ is important in the study of biLipschitz equisingularity. The case of curves is pretty well understood and described in \cite{BGG80}, \cite{Fe03}, \cite{GiSiSn22}, \cite{GiNoTe20}, \cite{NePi14},  \cite{Ph-Te20}. However not much is known in higher dimensions. In this work we begin the study of the Lipschitz saturation for toric singularities of arbitrary dimension.\\

Our study is inspired by the fact that the Lipschitz saturation of a germ of irreducible curve is a toric curve, i.e., is parametrized by monomials. Moreover, there exists an explicit procedure for calculating its corresponding numerical semigroup starting from the semigroup of the curve. Although there is no reason to believe that the Lipschitz saturation of a germ of arbitrary dimension has toric structure, we may still ask the following question: starting with a toric singularity, is the Lipschitz saturation also toric? If so, is there a procedure for computing the corresponding affine semigroup from the semigroup of the toric singularity? \\

We answer the first question in the situation that naturally generalizes the case of toric curves, that is, toric varieties having smooth normalization. Regarding the question of the description of the corresponding semigroup, we provide an answer for some families of toric singularities. 
An interesting semigroup theory aspect still lingers. The numerical semigroup of the Lipschitz saturation of a germ of toric curve can be calculated via the semigroup theoretic saturation as described in \cite{RG09}. One can ask: if we start with a toric germ $(X,0)$, is there an analogous operation on affine semigroups that calculates the affine semigroup of $(X^s,0)$? We conclude the paper by exhibiting examples showing that some properties of Lipschitz saturation of curves are no longer true in higher dimension. For instance, the embedding dimension has a different behavior.\\

The paper is divided as follows. In section \ref{secLips} we recall the general construction of Lipschitz saturation and its main properties. We also recall some known facts in the case of curves. Furthermore, we prove that if a germ $(Y,0)$ can be seen as the generic linear projection of another germ $(X,0)$, then their Lipschitz saturations are isomorphic (proposition \ref{ProyGen}). This proposition has important consequences, for instance, it allows us to build examples of  singularities that are not isomorphic to $(X,0)$ with toric Lipschitz saturation isomorphic to $(X^s,0)$ and toric normalization isomorphic to $(\overline{X},0)$.\\

In section \ref{ToricSing} we establish what we mean by a toric singularity $(X,0)$, which is basically taking the germ at the origin of an affine toric variety. We summarize some important facts about the passage from the algebraic toric variety to the germ of analytic space. The results of this section are essentially known (see, for instance, \cite{GP00b}). We include them here for the sake of completeness.\\
   
In section \ref{SatofTor} we prove our first main theorem: the Lipschitz saturation $(X^s,0)$ of a toric singularity whose normalization is smooth is again toric (theorem \ref{SatTor}). The key idea towards this result is that for every $f\in\Oxo^s$ all of its monomials are also in $\Oxo^s$. We believe this result is true for toric singularities in general. However, the smooth normalization hypothesis is there for technical reasons (remark \ref{Tech}). As a first application, we show that the Whitney Umbrella singularity is Lipschitz saturated (see example \ref{WU}).\\

In section \ref{TorSurf} we use the tools previously developed to give a combinatorial description of the Lipschitz saturation $(X^s,0)$ for some families of toric singularities. More concretely, given the semigroup $\Ga$ corresponding to the toric singularity $\Xo$, we explicitly describe the semigroup $\Ga^s$ of its Lipschitz saturation (see proposition \ref{ProdCurv} and theorem \ref{CalculoLip}). As a byproduct we can compute the embedding dimension of $(X^s,0)$ in these cases (see corollaries \ref{edim-mult prod} and \ref{edim-mult hypersurf}). We conclude by illustrating all these results in several examples.

\section{Lipschitz saturation of complex analytic germs}\label{secLips}

      The definition of Lipschitz saturation of a reduced complex analytic algebra $\Oxo$ is based on the concept of integral dependence on an ideal. Given an element 
   $r$ and an ideal $I$ in a ring $R$, we say that $\mathit{r}$  \textit{is integral over} $\mathit{I}$ if $r$ satisfies a relation of the form
   \[ r^m + a_1r^{m-1}+ a_2r^{m-2}+\cdots +a_{m-1}r + a_m=0,\]
   for some integer $m>0$, with $a_j \in I^j$  for $j=1,2,\ldots,m$.  The set $\overline{I}$ consisting of those elements of $R$ which are integral over $I$ is an ideal
   called the \textit{integral closure of} $I$ in $R$ (see \cite{HS06}, \cite{LT08}). \\
   
      
    Let us assume for simplicity that $(X,0)$ is irreducible and let $n^*: \Oxo \hookrightarrow \overline{\Oxo}$ be the integral closure of $\Oxo$ in its 
    field of fractions. Geometrically, this corresponds to the normalization map $n:(\overline{X},0) \to (X,0)$ and if we consider the holomorphic map
    \[ \left( \overline{X} \times_X \overline{X}, (0,0)\right) \hookrightarrow \left( \overline{X} \times \overline{X}, (0,0), \right)\]
   we get a surjective map of analytic algebras
   \[ \Psi: \overline{\Oxo} \widearc{\otimes}_\C \overline{\Oxo} \longrightarrow \overline{\Oxo} \widearc{\otimes}_{\Oxo} \overline{\Oxo},\]
   where $\widearc{\otimes}$ denotes the analytic tensor product which is the operation on the analytic algebras that corresponds to the fibre product
   of analytic spaces (for more details see \cite{Ada12} and \cite{GiNoTe20}).
   
   \begin{defi}\label{DefSat} Let $I_\Delta$  be the kernel of the morphism $\Psi$ above. We define the Lipschitz saturation $\Oxo^s$ of $\Oxo$ as the 
   algebra
   \[ \Oxo^s:= \left\{ f \in \overline{\Oxo} \, | \, f \widearc{\otimes}_\C 1 - 1\widearc{\otimes}_\C f  \in \overline{I_\Delta}\right\}.\]
   \end{defi}
   
   \begin{rem} \label{Rem1}A detailed discussion of the following facts can be found in \cite{Ph-Te20, GiNoTe20, Tei82}:
   \begin{enumerate}
   \item $\Oxo^s$ is an analytic algebra and coincides with the ring of germs of meromorphic functions on $(X,0)$ which are locally Lipschitz
           with respect to the ambient metric.
   \item We have injective ring morphisms
            \[ \Oxo \hookrightarrow \Oxo^s \hookrightarrow \overline{\Oxo}. \]
   \item The corresponding Lipschitz saturation map
             \[\zeta:(X^s,0) \to (X,0)\]      
             is a biLipschitz homeomorphism, induces an isomorphism outside the non-normal locus of $X$ and preserves the multiplicity, i.e.
             \[\textrm{mult }(X^s,0) = \textrm{ mult }(X,0).\]
             Moreover, it can be realized as a generic linear projection in the sense of \cite[Def. 8.4.2]{GiNoTe20}. 
   \item $\overline{\Oxo^s} =\overline{\Oxo}$ and the holomorphic map induced by $\Oxo^s \hookrightarrow \overline{\Oxo}$,
           \[\overline{X}\stackrel{n_s}{\longrightarrow} X^s, \]
          is the normalization map of $X^s$. Moreover the map
       \[n=\zeta \circ n_s : \overline{X} \to X\] is the normalization map of $X$.
   \end{enumerate} 
   \end{rem}

       Aside from these facts, little else is known in the general case. However, in the case of curves, the saturation has some very interesting equisingularity properties.
    First, Pham and Teissier proved that for a \emph{plane curve} $(X,0) \subset (\C^2,0)$ the Lipschitz saturation $\Oxo^s$ determines and is determined by the characteristic
    exponents of its branches and their intersection multiplicities, in particular \emph{the curve $(X^s,0)$ is an invariant of the equisingularity class of $(X,0)$} (See \cite{Ph-Te20}
    and \cite[Prop. VI.3.2]{BGG80}).  \\
    
      In the irreducible case (branches) they also prove that the \emph{the curve $(X^s,0)$ is always a toric curve,} (i.e parametrized by monomials), and if we have a parametrization of $(X,0) \subset (\C^2,0)$ 
         \[ t \longmapsto \left(\varphi_1(t),\varphi_2(t)\right); \hspace{0.2in} \varphi_1,\varphi_2 \in \C\{t\},\]
    then there is a simple procedure to calculate a parametrization of $(X^s,0)$. \begin{enumerate}
    \item Calculate the set of characteristic exponents $E:=\{\beta_0,\beta_1,\ldots,\beta_g\}$ of $(X,0)$.
    \item Calculate the smallest saturated numerical semigroup $\tilde{E} \subset \N$ containing $E$ as follows \cite[Chapter 3, Section 2]{RG09}:  
       \[ \widetilde{E_0}:= E \cup \beta_0 \cdot \N; \]
       \[ \widetilde{E_1}:= \widetilde{E_0} \cup \left\{ \beta_1 + ke_1 \,| \, k\in \N \right\}, \hspace{0.2in} e_1=\textrm{ gcd }\{\beta_0,\beta_1\};\]   
        \[ \widetilde{E_{j+1}}:= \widetilde{E_j} \cup \left\{ \beta_{j+1} + ke_{j+1} \,| \, k\in \N \right\}, \hspace{0.2in} e_{j+1}=\textrm{ gcd }\{e_j,\beta_{j+1}\};\]  
        \[ \tilde{E}=\widetilde{E_g}.\]    
    \item If $\{ a_1, \ldots, a_{\beta_0}\}$  is the minimal system of generators of $\tilde{E}$ then 
          \[ t \longmapsto \left(t^{a_1}, \ldots, t^{a_{\beta_0}}\right) \]
          is a parametrization of $(X^s,0)$.  Recall that in this setting $\beta_0$ is the multiplicity of $(X,0)$ and so we get that \emph{the embedding dimension of $(X^s,0)$ is equal to the multiplicity of $(X,0)$.}
    \end{enumerate}
      
        These results extend well beyond the plane curve case: in \cite[Thm VI.0.2, Prop. VI.3.1]{BGG80} the authors prove that for a curve $(X,0) \subset (\C^n,0)$,
        and a generic lineal projection $\pi: (\C^n,0) \to (\C^2,0)$ the Lipschitz saturations $\Oxo$ of $(X,0)$ and  $O_{\pi(X),0}^s$ of the plane curve 
        $(\pi(X),0) \subset (\C^2,0)$ are isomorphic. Even more, by \cite[Thm. 4.12]{GiSiSn22} we get that two germs of curves $(X,0)$ and $(Y,0)$ are bi-Lipschitz 
        equivalent if and only if their Lipschitz saturations are isomorphic.  In this sense \emph{the saturated curve $(X^s,0)$ can be seen as a canonical representative   
        of the bi-Lipschitz equivalence class of $(X,0)$.}

    \begin{exam}
       Let $(X,0) \subset (\C^3,0)$ be the space curve with normalization map
       \begin{align*} \eta: (\C,0) & \longrightarrow (X,0) \\ t &\longmapsto (t^6,t^{11}-t^9, t^{11}+t^9). \end{align*}
       For this curve, the projection on the first two coordinates is generic, giving us the plane curve
        \[ t \longmapsto (t^6,t^{11}-t^9)\]
        with characteristic exponents $E=\{6,9,11\}$. 
        
        Following the procedure described above, we get the saturated numerical semigroup $\tilde{E}$ with 
        minimal system of generators $\{6,9,11,13,14,16\}$. This determines a parametrization of the saturated curve $(X^s,0) \subset (\C^6,0)$ of the
        form:
        \begin{align*}    \eta^s:(\C,0) &\longrightarrow (X^s,0) \\
                                 \tau &\mapsto (\tau^6,\tau^9,\tau^{11},\tau^{13},\tau^{14},\tau^{16}).\end{align*}
        
  \end{exam} 
  
      Going back to the general case, we can prove that, just as in the case of curves, the Lipschitz saturation remains unchanged under generic
      linear projections.
                
     \begin{pro} \label{ProyGen}
      Let $(X,0) \subset (\C^n,0)$ be a germ of reduced and irreducible singularity, and let  $\pi: (\C^n,0) \to (\C^m,0)$
     be a generic linear projection with respect to $(X,0)$. Then $(X,0)$ and its image germ $(\pi(X),0)$ have isomorphic Lipschitz saturations, i.e.:
        \[ \Oxo^s \cong \mathcal{O}_{\pi(X),0}^s.\]
     \end{pro}
     
     Before going through the proof, recall that the cone $C_5(X,0)$, constructed by taking limits of bi-secants to $X$ at $0$, is an algebraic cone 
     defined by H. Whitney in \cite{Whi65}. A linear projection $\pi: (\C^n,0) \to (\C^m,0)$ with kernel $D$ is called $C_5$-general (or generic) 
     with respect to $(X,0)$ if it is transversal to the cone $C_5(X,0)$, meaning $D \bigcap C_5(X,0) = \{ 0\}$.  When $\pi$ is generic, it induces a 
     homeomorphism between $(X,0)$ and its image $(\pi(X),0)$, and these two germs have the same multiplicity, for a detailed explanation see \cite[Section 8.4]{GiNoTe20}.
      
     \begin{proof}(of proposition \ref{ProyGen}) \\ 
     
        After a linear change of coordinates we can assume that the linear projection $\pi:(\C^n,0) \to (\C^m,0)$ is the projection on the first $m$ coordinates
        $(z_1,\ldots,z_n) \mapsto (z_1,\ldots,z_m)$. Let $J \subset \mathcal{O}_{X \times X, (0,0)}$ denote the ideal defining the diagonal of $X \times X$ 
        \[ J= \left< z_1-w_1, \ldots, z_n-w_n\right>\mathcal{O}_{X \times X, (0,0)} \]
        Denote $J_\pi=\left<z_1-w_1,\ldots,z_m-w_m\right>\mathcal{O}_{X \times X, (0,0)}$. Then by \cite[Proposition 8.5.11]{GiNoTe20} the genericity of $\pi$ 
        is equivalent to the equality of integral closures $\overline{J}=\overline{J_\pi}$ in $\mathcal{O}_{X \times X, (0,0)}$.
        
       On the other hand, since $\pi:(X,0) \to (\pi(X),0)$ is a finite, generically 1-1 and surjective map, then the morphism
        \[\pi^*:\mathcal{O}_{\pi(X),0} \to \Oxo\]
        is injective, makes $\Oxo$ a finitely generated $\mathcal{O}_{\pi(X),0}$-module and  induces an isomorphism of the corresponding field of fractions
        $Q(\Oxo)$. All these together imply that $\Oxo$ and $\mathcal{O}_{\pi(X),0}$ have isomorphic integral closures in  $Q(\Oxo)$
        and so if $\eta:(\overline{X},0) \to (X,0)$ denotes the normalization of $(X,0)$ then the composition 
       \[ (\overline{X},0) \stackrel{\eta}{\longrightarrow} (X,0) \stackrel{\pi}{\longrightarrow} (\pi(X),0) \]
       is a normalization of $(\pi(X),0)$.
       
       Note that the ideal $I_\Delta$ of definition \ref{DefSat} is defined by the ``coordinate functions" of the normalization map
          \[  (\overline{X},0)  \stackrel{\eta}{\longrightarrow} (X,0)  \stackrel{\pi}{\longrightarrow} (\pi(X),0)\]
                       \[\underline{y} \mapsto \left( \eta_1(y), \ldots, \eta_n(y) \right)  \mapsto \left( \eta_1(y), \ldots, \eta_m(y) \right), \]      
      that is (see \cite[Section 8.5.2]{GiNoTe20}),
      \[ I_{\Delta_X} = \left< \eta_1(y) -\eta_1(x), \ldots, \eta_n(y)-\eta_n(x)\right>\mathcal{O}_{\overline{X} \times \overline{X}, (0,0)}   \]
        \[ I_{\Delta_{\pi(X)}} = \left< \eta_1(y) -\eta_1(x), \ldots, \eta_m(y)-\eta_m(x)\right>\mathcal{O}_{\overline{X} \times \overline{X}, (0,0)}.   \] 
To prove the desired result it is enough to prove the equality of the integral closures $\overline{I_{\Delta_X} } = \overline{I_{\Delta_{\pi(X)}} }$.
       But the germ map:
            \[ \eta \times \eta:\left( \overline{X} \times \overline{X}, (0,0)\right) \longrightarrow \left( X \times X, (0,0) \right) \]
       induces a morphism of analytic algebras:      
      \[ (\eta \times \eta)^* : \mathcal{O}_{X \times X, (0,0)} \longrightarrow \mathcal{O}_{\overline{X} \times \overline{X}, (0,0)} \]
      such that:
      \[ \left<  \left( \eta \times \eta\right)^*(J) \right> = I_{\Delta_X}\]
      \[ \left<  \left( \eta \times \eta\right)^*(J_\pi) \right> = I_{\Delta_{\pi(X)}} \]
       and since $\overline{J}=\overline{J_\pi}$ in $\mathcal{O}_{X \times X, (0,0)}$ the result follows. 
           \end{proof}
           
Note that in the course of the proof we have shown that the germ $(X,0)$ and its general projection $(\pi(X),0)$ have \emph{isomorphic normalizations}. On the downside, it will no longer be true in general that bi-Lipschitz equivalent germs will have isomorphic Lipschitz saturations. This is because a germ $(X,0)$ and its Lipschitz saturation $(X^s,0)$ always have the same multiplicity, however in \cite{BFSV20} the authors prove that in dimension bigger than two, multiplicity of singularities is not a bi-Lipschitz invariant. 
     
\section{Toric singularities} \label{ToricSing}

In this section we establish what we mean by a toric singularity. We also introduce the notation we use regarding toric varieties.

Let $\A=\{\ga_1,\ldots,\ga_n\}\subset\Z^d$, $\Ga=\N\A=\big\{\sum_i a_i\ga_i|a_i\in\N\big\}$, and $\sd=\R_{\geq0}\A$. Assume that the group generated by $\A$ is $\Z^d$ and that $\sd$ is a strongly convex cone. Consider the following homomorphism of semigroups, 
\begin{align}
\pi:&\N^n\to\Ga\notag\\
&\alpha\mapsto \sum_i \alpha_i\ga_i,\notag
\end{align}
and the induced $\C-$algebra homomorphism,
\begin{align}
\vp:&\pol\to\C[t_1^{\pm},\ldots,t_d^{\pm}]\notag\\
&\hspace{1.5cm}z_i \mapsto t^{\ga_i}=t_1^{\ga_{i,1}}\cdots t_d^{\ga_{i,d}}\notag\\
&\hspace{1.5cm}z^{\alpha} \mapsto t^{\pi(\alpha)}.\notag
\end{align}
Let $\ig=\ker\vp$. Recall that $\ig$ is a prime ideal and $\ig=\langle z^{\alpha}-z^{\beta}|\pi(\alpha)=\pi(\beta)\rangle.$
Let $X\subset\C^n$ be the affine variety defined by $\ig$. Then $X$ is a $d$-dimensional affine toric variety containing the origin. Let $\C[X]$ be the ring of regular functions on $X$ and $\pg$ the $\C$-algebra of the semigroup $\Ga$. Recall that $\C[X]\cong\pg=\pgg\cong\pol/\ig$ \cite[Chapter 1]{CLS11}.

 Next we discuss some basic results on the passage from the algebraic toric variety $X$ to the germ of analytic space $\Xo$. Throughout this section we use the following notation.

\begin{itemize}
\item $\fsg$ denotes the ring of formal power series with exponents in $\Ga$.  
\item $\csg$ denotes the subring of $\fsg$ consisting of convergent series in a neighborhood of $0\in X$. 
\item $\Oxo$ denotes the algebra of germs of holomorphic functions on $0\in X$.
\end{itemize}

\begin{rem}
Notice that $\fsg$ is indeed a ring since $\Ga$ is contained in a strongly convex cone which implies that every element of $\Ga$ can be written as a sum of elements of $\Ga$ in finitely many different ways.
\end{rem}

\begin{lem}\label{lema P}
With the previous notation, 
$$\Oxo\cong\csg\cong\cs/\ig\cs.$$
\end{lem}
\begin{proof}
The first isomorphism is proved in \cite[Lemme 1]{GP00b} or  \cite[Lemma 1.1]{GP00}. The proof given there is for normal toric varieties. However, the same proof holds in the non-normal case.

We prove the second isomorphism. Consider the following exact sequence:
$$\xymatrix{0\ar[r]&\ig\ar[r]&\pol\ar[r]^{\vp}&\pg\ar[r]&0}.$$
Let $\m=\langle z_1,\ldots,z_n \rangle$. Taking completions with respect to $\m$ we obtain the following exact sequence:
$$\xymatrix{0\ar[r]&\hat{\ig}\ar[r]&\fs\ar[r]^{\vp_f}&\fsg\ar[r]&0}.$$
On the other hand, $\hat{\ig}\cong \ig\fs$ \cite[Proposition 10.15]{AM}. Denote $\vp_a=\vp_f|_{\cs}$. In the proof of \cite[Lemma 1.1]{GP00},  alternatively \cite[Lemme 1]{GP00b} it is shown that $\mbox{Im }\vp_a=\csg$. It remains to prove that $\ker\vp_a=\ig\cs$.

Since $\ker \vp_f=\ig\fs$, it follows that $\ig\cs\subset\ker\vp_a$. Let $F\in\cs$ be such that $0=\vp_a(F)=\vp_f(F)$. We conclude that $F\in\cs\cap\ig\fs=\ig\cs$ (these ideals are equal by \cite[Exercise 8.1.5]{dJP}).
\end{proof}

\begin{lem}\label{normaliz}
Let $\bar{X}$ be the (algebraic) normalization of $X$. Then $(\bar{X},0)$ is the (analytic) normalization of $\Xo$. Moreover, $\Onxo\cong\C\{\sd\cap\Z^d\}.$
\end{lem}
\begin{proof}
Let $\eta:\bar{X}\to X$ be the normalization. Recall that $\bar{X}$ is the toric variety defined by the semigroup $\sd\cap\Z^d$  and that $\eta$ is induced by the inclusion of semigroups $\Ga\subset\sd\cap\Z^d$ \cite[Proposition 1.3.8]{CLS11}. In particular, $\eta$ is a toric morphism. 

Being the normalization, $\eta$ is an isomorphism on dense open sets. In addition, $\eta^{-1}(0)=0$. Indeed, let $q\in\bar{X}$ be such that $\eta(q)=0$. Recall that points in toric varieties correspond to homomorphisms of semigroups. Hence, the homomorphism corresponding to $\eta(q)$ sends every non-zero element of $\Ga$ to 0. On the other hand, for every $m\in\sd\cap\Z^d$ there is $k\geq1$ such that $km\in\Ga$. Since $\eta$ is induced by the inclusion $\Ga\subset\sd\cap\Z^d$, it follows that $q=0$.


By the previous paragraph, the induced germ of an analytic function $\eta:(\bar{X},0)\to \Xo$ is finite and generically 1-1. On the other hand, it is known that $\bar{X}$ normal implies that $(\bar{X},0)$ normal \cite[Satz 4]{K}. By the uniqueness of normalization we conclude that $\eta:(\bar{X},0)\to \Xo$ is the normalization of $\Xo$.

Finally, $\Onxo\cong\C\{\sd\cap\Z^d\}$ follows using lemma \ref{lema P}.
\end{proof}

\begin{coro}
$\Xo$ is irreducible as a germ.
\end{coro}
\begin{proof}
By a well-known theorem of Zariski \cite{Z48}, $\bar{X}$ irreducible and normal at $0$ implies that the completion of the local ring at 0 is an integral domain. Hence, $(\bar{X},0)$ is irreducible by lemma \ref{lema P}. Hence, the germ $(X,0)$ must also be irreducible.
\end{proof}

\begin{defi}
Let $\Xo$ be a germ of an analytic space. We say that $\Xo$ is a germ of a toric singularity if there exists a finitely generated semigroup $\Ga\subset\Z^d$ contained in a strongly convex cone such that $\Oxo\cong\csg$. Equivalently, let $\xg$ be the affine toric variety defined by $\Ga$. Then $\Xo$ is a toric singularity if it is isomorphic, as germs, to $(\xg,0)$.
\end{defi}

\begin{exam}
Let $X\subset\C^n$ be a toric variety containing the origin. Then $\Xo$ is a germ of a toric singularity by lemma \ref{lema P}.
\end{exam}

\section{Lipschitz saturation of toric singularities} \label{SatofTor}

      We have seen that for a toric singularity the analytic algebra $\Oxo$ is generated by monomials. With this in mind, the main idea to prove  that the Lipschitz saturation of a toric singularity is again toric, consists of showing that every monomial of an element of the saturation also belongs to the saturation. The first step towards that goal is the study of the integral closure of homogeneous ideals in the ring of power series.

   It is well known that the integral closure of a homogeneous ideal in a polynomial ring is again a homogeneous ideal (\cite[Prop. (f), pg 38]{Vas05}). However
   we need this to be true for ideals generated by homogeneous polynomials in power series rings. That is the content of the following proposition, whose proof
   was kindly communicated to us by Professors Irena Swanson and Craig Huneke. 
   
   Note that  an ideal $\mathcal{I}\subset \C\{z_1,\ldots,z_n\}$ is generated by homogeneous polynomials if and only if $f \in \I$ implies that 
   every homogeneous component of $f$ is also in $\I$.

\begin{pro}\label{CerraduraEnteraHomogenea}
    Let $\mathcal{I}\subset \C\{z_1,\ldots,z_n\}$ be an ideal generated by homogeneous polynomials in the ring of convergent power series. 
  Then its integral closure $\overline{\mathcal{I}}$ is also generated by homogeneous polynomials.     
\end{pro}
\begin{proof}
   Let $R:=\C\{z_1,\ldots,z_n\}$ denote the ring of convergent power series in $n$-variables over the field of complex numbers with maximal 
   ideal $\m=\left<z_1,\ldots,z_n\right>$, and let $\I=\left<f_1,\ldots,f_m\right>R$ where each $f_j \in A:=\C[z_1,\ldots,z_n] \subset R$ is a homogeneous polynomial. \\
   
    Let $K=\left<f_1,\ldots, f_m\right>A$  be the ideal generated by the $f_j$'s in the polynomial ring $A$, so $\I=\left<K\right>R$. Suppose that $\I$ contains a power of $\m$, i.e. there exists $k\geq1$ such that $\m^k \subset \I$, and let $s \in R$ be in the integral closure of $\I$. We can write 
  \[s = p + s',\]
  where $p\in A$ is a polynomial and $s' \in \m^k \subset \I \subset \overline{\I}$ is a series of order greater than or equal to $k$. Note that $p=s-s' \in \overline{\I}$.
  
  Since the ring extension $A_\m \rightarrow R$ is faithfully flat (\cite[Lemma B.3.4]{GLS07}) then $\overline{\I} \cap A_\m= \overline{\left<K\right>A_\m}$ (\cite[Prop.1.6.2]{HS06}), in particular $p/1 \in \overline{\left<K\right>A_\m}$. This means that in the localized polynomial ring $A_\m$ we have an equation of integral 
  dependence of the form:
  \[ \frac{p^r}{1}+ \frac{b_1}{c_1} \frac{p^{r-1}}{1} + \cdots + \frac{b_{r-1}}{c_{r-1}}\frac{p}{1} + \frac{b_r}{c_r} = \frac{0}{1},\]
 where $b_j \in K^j$ and $c_j \notin \m$. Letting $u=c_1\cdots c_r \in A$ and multiplying this equation by the unit $\frac{u^r}{1}$ of $A_\m$ we get the 
 following equality in $A_\m$:
 \[  \frac{(up)^r + \widetilde{b_1} (up)^{r-1}+ \cdots + \widetilde{b_{r-1}}(up) + \widetilde{b_r} }{1}=\frac{0}{1}.\]
Since A is an integral domain we get an integral dependence equation in $A$
 \[(up)^r + \widetilde{b_1} (up)^{r-1}+ \cdots + \widetilde{b_{r-1}}(up) + \widetilde{b_r}=0,\]
 that is, $up \in \overline{K} \subset A$. The inclusions $\m^k \subset K \subset \overline{K}$ imply that $\overline{K}$ is an $\m$-primary
 ideal of $A$ and since $u \notin \m$ then $p \in \overline{K}$. Since $K$ is a homogeneous ideal in the polynomial ring $A$ then 
 $\overline{K}$ is also homogeneous (\cite[Prop. (f), pg 38]{Vas05}) and so each homogeneous component of $p$ is also in $\overline{K} \subset \overline{\I}$. This
 implies that $\overline{\I}$ is also generated by homogeneous polynomials. 
 
   Now let $\I \subset R$ be an arbitrary ideal generated by homogeneous polynomials, and  let $s$ be in the integral closure $\overline{\I}$ as before. Let $s_0$ be the initial form of $s$, $s_0$ is the non-zero homogeneous component of $s$ of lowest degree. For all $j \in \N$, $s$ is in the integral closure of $\I + \m^j$.  We proved in the previous paragraph that the integral closure
  $\overline{\I + \m^j}$  of $\I + \m^j$ in $R$ is generated by homogeneous polynomials for all $j \in \N$.  In particular, $s_0$ is in the integral closure  $\overline{\I + \m^j}$  for all $j$. But by 
  \cite[Corollary 6.8.5]{HS06} this implies that $s_0$ is in the integral closure $\overline{\I}$. Now we start over with $s'=s-s_0 \in \overline{\I}$ and in this way we get that all homogeneous components of $s$ are in $\overline{\I}$ which is what we wanted to prove. 
\end{proof}

\begin{teo} \label{SatTor}
  Let $(X,0)$ be a $d$-dimensional toric singularity with smooth normalization. Then the Lipschitz saturated germ $(X^s,0)$ is also a toric singularity. 
\end{teo}

We will do this proof in several steps starting with the following lemma.

\begin{lem}\label{Monomial}
  Let $(X,0)$ be a $d$-dimensional toric singularity with smooth normalization and let $f \in \overline{\Oxo}$ be a homogeneous polynomial such that $f \in \Oxo^s$. 
  Then every monomial of $f$ is also in $\Oxo^s$.
\end{lem}
\begin{proof}
Let $\Ga\subset\Z^d$ be the semigroup defining $(X,0)$. Recall that $\sd=\R_{\geq0}\Ga$ is a strongly convex cone. This fact, together with the condition of smooth normalization allows us to assume, up to a change of coordinates, that $\Ga\subset\N^d$ and $\sd\cap\Z^d=\N^d$ (see lemma \ref{normaliz}).

Let $\mathcal{A}=\{a_1,\ldots,a_n\} \subset \N^d$ be the minimal generating set of $\Gamma$ defining the toric singularity $(X,0)$. The normalization map 
  can be realized as the monomial morphism:
   \begin{align*} n: (\C^d,0) &\longrightarrow (X,0) \\
   \left(u_1,\ldots,u_d\right) & \mapsto  \left( u^{a_1}, \ldots, u^{a_n} \right).\end{align*}
The ideal $I_\Delta$ of definition \ref{DefSat} is a homogeneous, binomial ideal in the ring of convergent power series $\C\{x_1,\ldots, x_d, y_1,\ldots,y_d\}$:
   \[I_\Delta= \left<  x^{a_1}-y^{a_1}, \ldots, x^{a_n}-y^{a_n}\right>. \]
   For any point $\tau=(t_1,\ldots,t_d) \in \left(\C^*\right)^d$ we have an automorphism 
    \[ \varphi_\tau:\C\{x_1,\ldots, x_d, y_1,\ldots,y_d\}\circlearrowleft \] defined by  $x_j \mapsto t_jx_j$ and
   $y_j \mapsto t_jy_j$,  such that $\varphi_\tau\left( I_\Delta \right)=I_\Delta$.
   
   Now let $f \in \overline{\Oxo} \cong \C\{u_1,\ldots,u_d\}$ be a homogeneous polynomial such that $f \in \Oxo^s$. Then $f(x)-f(y) \in \overline{I_\Delta}$ and it satisfies an integral 
   dependence equation of the form:
   \[ \left( f(x) - f(y)\right)^m + h_1(x,y)\left( f(x) - f(y)\right)^{m-1} + \cdots + h_m(x,y)=0,\]
   where $h_j(x,y) \in I_\Delta^j$. By applying the morphism $ \varphi_\tau$ to the previous equation we get: 
   \[ \left( f(\tau x) - f(\tau y)\right)^m + h_1(\tau x,\tau y)\left( f(\tau x) - f(\tau y)\right)^{m-1} + \cdots + h_m(\tau x,\tau y)=0, \]
    where $h_j(\tau x,\tau y) \in I_\Delta^j$, and so we get that $g_\tau:=f(t_1u_1,\ldots,t_du_d) \in \Oxo^s$. 
    
    Since $f$ is a homogeneous polynomial, say of order $k$, then we can write it in the form
    \[ f(u_1,\ldots,u_d)= \sum_{\alpha_1+\cdots+\alpha_d=k} b_{\alpha} u^{\alpha},\]
    where $b_\alpha \in \C$ and $\alpha=(\alpha_1,\ldots,\alpha_d)$. Then we have an expression for $g_\tau$ of the form
    \begin{align*} g_\tau=f(t_1u_1,\ldots,t_du_d)&=\sum_{\alpha_1+\cdots+\alpha_d=k}\tau^\alpha b_{\alpha} u^{\alpha} \\
      &= \left< \left(\tau^{\alpha^1},\ldots, \tau^{\alpha^N} \right), \left(b_{\alpha^1} u^{\alpha^1},\ldots, b_{\alpha^N} u^{\alpha^N}\right) \right>. \end{align*}
   By choosing $\tau_1, \ldots, \tau_N$ generic points in $\left(\C^*\right)^d$ we get
   \[ \begin{pmatrix} g_{\tau_1} \\ \vdots \\g_{\tau_N} \end{pmatrix} = \begin{pmatrix} 
       \tau_1^{\alpha^1} & \ldots &  \tau_1^{\alpha^N} \\  \vdots & \cdots & \vdots \\ \tau_N^{\alpha^1} & \ldots &  \tau_N^{\alpha^N} \end{pmatrix} 
       \begin{pmatrix} b_{\alpha^1} u^{\alpha^1} \\ \vdots \\ b_{\alpha^N} u^{\alpha^N} \end{pmatrix},  \]
    where the $i$-th row of the matrix  corresponds to the image of  the point $\tau_i \in \left(\C^*\right)^d $ of the Veronese map $\nu_k: \mathbb{P}^{d-1} \to \mathbb{P}^{N-1}$ 
    of degree $k$. Since the image of the Veronese map is a nondegenerate projective variety, in the sense that it is not contained in any hyperplane, then these $N$ points 
    are in general position. This implies that the matrix is invertible, and so we have 
    \[ \begin{pmatrix} b_{\alpha^1} u^{\alpha^1} \\ \vdots \\ b_{\alpha^N} u^{\alpha^N} \end{pmatrix}= 
    \begin{pmatrix} 
       \tau_1^{\alpha^1} & \ldots &  \tau_1^{\alpha^N} \\  \vdots & \cdots & \vdots \\ \tau_N^{\alpha^1} & \ldots &  \tau_N^{\alpha^N} \end{pmatrix} ^{-1}
       \begin{pmatrix} g_{\tau_1} \\ \vdots \\g_{\tau_N} \end{pmatrix} \] 
   In particular we have that $ b_{\alpha^j} u^{\alpha^j} \in \Oxo^s$ which is what we wanted to prove.         
\end{proof}

\begin{proof} (\textit{of theorem \ref{SatTor}} ) \\
   We want to prove that there exists a finitely generated semigroup $\Gamma^s \subset \N^d$ such that \[\Oxo^s \cong \C\{\Gamma^s\}.\]
   As we mentioned before, in this setting the ideal $I_\Delta$ is a homogeneous, binomial ideal in the ring of convergent power series $\C\{x_1,\ldots, x_d, y_1,\ldots,y_d\}$,
   and by proposition \ref{CerraduraEnteraHomogenea} the ideal $\overline{I_\Delta}$ is also generated by homogeneous polynomials. This means that if 
   a series $f\in \C\{u_1\ldots,u_d\}$ is in $\Oxo^s$,
   \[f= f_m + f_{m+1}+ \cdots + f_N + \cdots ,\] then every homogeneous component $f_j$ of $f$ is in $\Oxo^s$. But by lemma \ref{Monomial} this implies that every 
   monomial of degree $j$ with a non-zero coefficient in $f_j$ is also in $\Oxo^s$.  Let $\Gamma^s$ be the semigroup defined by
   \[ \Gamma^s=\{ \alpha \in \N^d \, | \, u^\alpha \in \Oxo^s \}.\]
   We have that 
   \[ \Oxo^s \subset \C\{\Gamma^s\}\subset \C\{u_1,\ldots,u_d\}.\]
   We have to prove that $\Gamma^s$ is finitely generated and the equality $\Oxo^s = \C\{\Gamma^s\}$. \\
   
   To begin with the latter, take $g\in \C\{\Gamma^s\}$. For every $k\geq 0$, write $g= g_{\leq k} + \widetilde{g_k}$
   where $g_{\leq k}$ is the truncation of $g$ to degree $k$, and since it is a finite sum of monomials of $\Oxo^s$, by definition of $\Gamma^s$,
    we have that $g_{\leq k} \in \Oxo^s$. This means that
   \[ g_{\leq k}(x) -g_{\leq k}(y) \in \overline{I_\Delta}\] 
   and since $\widetilde{g_k}(x)-\widetilde{g_k}(y) \in \m^{k+1}\C\{x,y\}$ we have that
   \[ g(x) -g(y) \in \overline{I_\Delta} + \m^{k+1} \subset \overline{I_\Delta + \m^{k+1}}.\]
   In particular,
   \[ g(x)-g(y) \in \bigcap_{k \in \N}\overline{I_\Delta + \m^{k+1}} = \overline{I_\Delta} \textrm{ by \cite[Corollary 6.8.5]{HS06}}\]
   then $g \in \Oxo^s$ and we have the equality we wanted. 

Now we prove that $\Ga^s$ is a finitely generated semigroup. Let $\leq$ be a monomial order in $\C[u_1,\ldots,u_d]$. We order the elements of $\Ga^s$ and write them as $\{\beta^j\, | \, j \in \N \}$, where  $0=\beta^0 < \beta^1 < \beta^2 <\cdots$. Consider the following ascending chain of ideals in $\Oxo^s$:
$$\langle u^{\beta^1} \rangle\subset \langle u^{\beta^1},u^{\beta^2} \rangle \subset \langle u^{\beta^1},u^{\beta^2},u^{\beta^3} \rangle \subset \cdots$$
Since $\Oxo^s$ is Noetherian, we have $u^{\beta^j}\in\langle u^{\beta^1},\ldots,u^{\beta^m} \rangle$, for some $m\in\N$ and for all $j\in\N$. We claim that $\Ga^s=\N(\beta^1,\ldots,\beta^m)$.

Indeed, let $u^{\beta^j}=\sum_{i=1}^m F_iu^{\beta^i}$, for some $F_i\in\Oxo^s$. Since $\Oxo^s\subset\C\{u_1,\ldots,u_d\}$ it follows that $u^{\beta^j}=u^{\beta}u^{\beta^i}$, for some $i$ and some monomial $u^{\beta}$ of $F_i$. As before, we have $\beta\in\Ga^s$. Hence, $\beta^j$ is the sum of $\beta^i$ plus an element $\beta\in\Ga^s$ and $\beta<\beta^j$. Continuing this way, we obtain that $\beta^j\in\N(\beta^1,\ldots,\beta^m)$.
\end{proof}

\begin{rem}\label{Tech} The smooth normalization hypothesis plays a key role in the proofs of lemma \ref{Monomial} and theorem \ref{SatTor} since it implies that  the ideal $I_\Delta$  is a homogeneous ideal in the ring of convergent power series $\C\{x_1,\ldots, x_d, y_1,\ldots,y_d\}$ and allows us to use proposition \ref{CerraduraEnteraHomogenea}. 
\end{rem} 

We now know that for a germ $(X,0)$ of toric singularity with smooth normalization and associated semigroup $\Ga$, the Lipschitz saturated germ $(X^s,0)$ is again a toric singularity with associated semigroup $\Ga^s$. In this setting we have $\Ga \subset \Ga^s \subset \N^d$ and we need to determine which elements $ \alpha \in \N^d$ we have to add to $\Ga$ in order to obtain $\Ga^s$. Since many properties of a toric variety are encoded in its semigroup, we can use them to start discerning. This is the content of the following proposition.  
  
  \begin{pro}\label{NoAgrego} Let $(X,0)$ be a germ of $d$-dimensional toric singularity with smooth normalization. Let $\Ga \subset \N^d$ be the associated 
  semigroup and $K_+(\Ga)$ the convex hull of $\Ga\setminus\{0\}$ in $\R^d$. If $ \alpha \in \N^d \setminus K_+(\Ga)$ then $\alpha \notin \Ga^s$.
  \end{pro}
  \begin{proof}
     By \cite[Chapter 5, theorem 3.14]{GKZ08}, the multiplicity of a toric germ is determined by the (normalized) volume of the complement of $K_+(\Ga)$ in $\N^d$. 
     But we know from \ref{Rem1} that a germ $(X,0)$ and its Lipschitz saturation $(X^s,0)$ have the same multiplicity, and so we must have that
     $K_+(\Ga)=K_+(\Ga^s)$ which finishes the proof. 
      \end{proof}

    In the case of curves, this means that if $m$ is the minimal non-zero element of $\Ga \subset \N$ then $k\in \N$, with $ k< m$ implies that $k \notin \Ga^s$ (see section \ref{secLips}). Recall that in this case $m$ is equal to the multiplicity of the curve. 
    
    \begin{exam}\label{WU}The Whitney Umbrella. \\
       The surface singularity $(X,0) \subset (\C^3,0)$ defined by the equation $y^2-x^2z=0$ is a toric singularity with smooth normalization
       given by
       \[ (u,v) \mapsto (u, uv,v^2)\] 
       and associated semigroup $\Ga \subset \N^2$ with minimal generating set $\{(1,0),(1,1),(0,2)\}$. This translates 
       to $\Oxo \cong \C\{u,uv,v^2\} \subset \C\{u,v\}$ and a point $(a,b) \in \Ga^s$ is identified with the monomial $u^av^b \in \Oxo^s \subset \C\{u,v\}$.
       
       \begin{figure}[H]
       \centering
       \includegraphics[width=\textwidth]{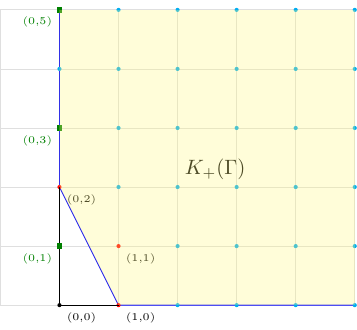}
       \caption{The square points are not in $\Ga$.}
       \end{figure}
       
             Note that $\N^2\setminus \Ga=\left\{ (0,2k+1) \,|\, k\in \N\right\}$. We will show that none of them are in $\Ga^s$. Hence, $\Ga=\Ga^s$. In the proof of theorem \ref{SatTor} we showed that $\Oxo^s=\C\{\Ga^s\}$. We conclude that $\Oxo=\Oxo^s$ and so the Whitney Umbrella coincides with its Lipschitz saturation $(X^s,0)$. 


         To begin with, the only point in $\N^2 \setminus K_+(\Ga)$ is $(0,1)$, and so $(0,1) \notin \Ga^s$ by the previous proposition. For the points
         of the form $(0,r)$ with $r>1$ odd, consider the ideal 
        \[ I_\Delta=\left< x_1-y_1, x_1x_2-y_1y_2,x_2^2-y_2^2\right> \C\left\{ x_1,x_2,y_1,y_2\right\} \]  
         Taking the arc $\varphi:(\C,0) \to \left(\C^2 \times \C^2, 0\right)$ defined by $t \mapsto (t^{r+1},t,t^{r+2},-t)$
         we have the corresponding morphism of analytic algebras $\varphi^*: \C\left\{ x_1,x_2,y_1,y_2\right\} \to \C\{t\}$ 
         such that
         \[ \varphi^*(x_2^r -y_2^r)=2t^r \notin \left<\varphi^*(I_\Delta)\right>=\left<t^{r+1}\right>.\]
         By \cite[Thm 2.1]{LT08} this implies that $x_2^r -y_2^r \notin \overline{I_\Delta}$, i.e. $v^r \notin \Oxo^s$. 
        
           \end{exam}
  
\section{Some examples.} \label{TorSurf}
 
  In this section we will show how to calculate the Lipschitz saturation of some families of toric singularities, starting with products of curves.\\
  
  Let $(X_1,0)$ and $(X_2,0)$ be two germs of toric singularities of dimension $1$ defined by the semigroups $\Ga_1$ and $\Ga_2$
  with corresponding minimal generating sets $\A_1=\{\ga_1,\ldots,\ga_m\}$ and $\A_2=\{\omega_1,\ldots, \omega_n\}$.
  The germ $(X,0)=(X_1 \times X_2,0)$ is a toric surface singularity with semigroup $\Ga \subset \N^2$ generated by 
  $\A=\left\{(\ga_i,0), (0,\omega_j)\right\}_{i,j}$, that is $\Ga=\Ga_1 \times \Ga_2$. Note that the normalization of $(X,0)$ is smooth and the normalization map can be written as
  \begin{align*}
    \eta: (\C^2,0) & \longrightarrow (X,0) \subset (\C^{m+n},0) \\
                 (u,v) &\mapsto \left( u^{\ga_1},\ldots,u^{\ga_m},v^{\omega_1},\ldots, v^{\omega_n}\right).
  \end{align*}
  
  \begin{pro}\label{ProdCurv} For a germ of surface singularity $(X,0)=(X_1 \times X_2,0) $ defined by a product of toric curves, the Lipschitz saturation $(X^s,0)$
  is a toric surface singularity with semigroup
  \[\Ga^s =\Ga_1^s \times \Ga_2^s,\]
  where $\Ga_1^s$ and $\Ga_2^s$ are the semigroups of the Lipschitz saturated curves $(X_1^s,0)$ and $(X_2^s,0)$ described in section \ref{secLips}.
  \end{pro}
  
  \begin{proof}
    We know that $\Oxo^s\subset \C\{u,v\}$ is an analytic algebra generated by monomials, so we need to characterize them. A monomial $u^\alpha v^\beta \in \Oxo^s$ 
    defines a meromorphic function on a neighborhood $U$ of the origin in $X$ which is locally Lipschitz with respect to the ambient metric. If we consider the normalization 
    of $(X,0)$ as before 
     \begin{align*}
    \eta: (\C^2,0) & \longrightarrow (X,0) \subset (\C^{m+n},0) \\
                 (u,v) &\mapsto \left( u^{\ga_1},\ldots,u^{\ga_m},v^{\omega_1},\ldots, v^{\omega_n}\right),
  \end{align*}
  then for any sufficiently small $v_0$ the restriction of $u^\alpha v^\beta$ to $\left(X_1\times \{v_0^\omega\},(0,v_0^\omega)\right)$ 
  tells us that $u^\alpha v_0^\beta$ defines a meromorphic locally Lipschitz function on $(X_1,0)$ and so $u^\alpha \in \mathcal{O}_{X_1,0}^s$; equivalently
  $\alpha \in \Ga_1^s$. The same reasoning with the restriction to $\left( \{u_0^\gamma\} \times X_2, (u_0^\ga,0)\right)$ tells us that $\beta \in \Ga_2^s$
  and so $(\alpha,\beta) \in \Ga_1^s \times \Ga_2^s$.

    On the other hand, in this setting we have the ideal $I_\Delta$ defined by 
   \[I_\Delta= \left< x_1^{\ga_1}-y_1^{\ga_1},\ldots, x_1^{\ga_m}-y_1^{\ga_m},x_2^{\omega_1}-y_2^{\omega_1}, \ldots,
       x_2^{\omega_n}-y_2^{\omega_n} \right>\C\{x_1,x_2,y_1,y_2\}.\]
    Let $\alpha \in \Ga_1^s$ then by definition we have 
    \[ x_1^\alpha - y_1^\alpha \in \overline{\left<x_1^{\ga_1}-y_1^{\ga_1},\ldots,x_1^{\ga_m}-y_1^{\ga_m}\right>}\C\{x_1,y_1\}.\]
    In particular, $x_1^\alpha - y_1^\alpha \in \overline{I_\Delta}$ and so $u^\alpha \in \Oxo^s \subset \C\{u,v\}$. Analogously
    for every $\beta \in \Ga_2^s$ we have $v^\beta \in \Oxo^s$. Since $\Oxo^s$ is an analytic algebra, this implies that for every 
    $\alpha \in \Ga_1^s$ and $\beta \in \Ga_2^s$  the monomial $u^\alpha v^\beta \in \Oxo^s$ which finishes the proof. 
  \end{proof}
  
\begin{coro}\label{edim-mult prod}
Let $(X,0)=(X_1 \times X_2,0) $ be a product of toric curves. Then $\edim (X^s,0)=\edim (X_1^s,0)+\edim (X_2^s,0)$. In addition, $\mult (X^s,0)=\mult (X_1^s,0)\cdot\mult (X_2^s,0)$. 
\end{coro}
\begin{proof}
First, recall that the embedding dimension of the origin of a toric variety i.e., the dimension of its Zariski tangent space, coincides with the cardinality of the minimal generating set of the correspondig semigroup. On the other hand, the multiplicity at the origin of a toric variety is also described combinatorially in terms of the semigroup (see proposition \ref{NoAgrego}). Hence, both assertions follow from proposition \ref{ProdCurv}.
\end{proof}

\begin{rem}
Notice that both proposition \ref{ProdCurv} and corollary \ref{edim-mult prod} hold for the product of any finite number of toric curves, with the same proof.
\end{rem}

  \begin{exam}
    Starting from the space curves $(X_1,0)$ and $(X_2,0)$ parametrized respectively by
    \[ u \longmapsto (u^4,u^6,u^7), \hspace{1in} v \longmapsto (v^6,v^9,v^{11}),\]
    we get the toric surface $(X,0)\subset (\C^6,0)$ of multiplicity $24$ and embedding dimension $6$ defined by the ideal 
    \[ I_X=\left< y^2-x^3, c^3-a^4b,b^2-a^3, z^2-x^2y \right>\C\{x,y,z,a,b,c\}. \]
    The normalization map is given by:
      \begin{align*}
    \eta: (\C^2,0) & \longrightarrow (X,0) \subset (\C^{6},0) \\
                 (u,v) &\mapsto \left( u^4,u^6,u^7,v^6,v^9, v^{11}\right).
  \end{align*}
  Following the procedure described in section \ref{secLips} we obtain that the semigroup $\Ga_1^s \subset \N $ is generated by $\mathcal{A}_1=\{4,6,7,9\}$ 
  and the semigroup $\Ga_2^s \subset \N$ is generated by $\mathcal{A}_2=\{6,9,11,13,14,16\}$. By proposition \ref{ProdCurv} the Lipschitz saturation $(X^s,0) \subset 
  (\C^{10},0)$  is the toric singularity defined by the semigroup $\Ga^s \subset \N^2$ generated by the set
  \[ \mathcal{A}= \left\{(4,0),(6,0),(7,0),(9,0),(0,6),(0,9),(0,11),(0,13),(0,14),(0,16) \right\},\]
and with normalization map given by:
      \begin{align*}
    \eta: (\C^2,0) & \longrightarrow (X^s,0) \subset (\C^{6},0) \\
                 (u,v) &\mapsto \left( u^4,u^6,u^7,u^9,v^6,v^9, v^{11},v^{13},v^{14},v^{16}\right).
  \end{align*}
    $(X^s,0)$ is a toric germ of multiplicity $24$ and embedding dimension $10$. \\

       Since $(X,0)\subset (\C^6,0)$ is a germ of singular surface, we have that the cone $C_5(X,0)$ is of dimension $3$ or $4$ and so almost every linear projection 
       \begin{align*} \pi: (\C^6,0) &\longrightarrow (\C^4,0) \\
                             \begin{pmatrix} z_1 \\  \vdots \\ z_6 \end{pmatrix} &\longmapsto \begin{pmatrix} a_{11} & \cdots & a_{16} \\ \vdots & \vdots & \vdots \\
                            a_{41} & \cdots & a_{46} \end{pmatrix}  \begin{pmatrix} z_1 \\  \vdots \\ z_6 \end{pmatrix}
       \end{align*}
       is generic, in the sense that it is transversal to this cone (see \cite[Prop. 8.4.3]{GiNoTe20}). Using proposition \ref{ProyGen}, for each such $\pi$ we get a 
       germ of singular surface $(Y_{\pi},0):=(\pi(X),0) \subset (\C^4,0)$ of multiplicity $24$ and embedding dimension at most $4$, with normalization map (see the proof of proposition \ref{ProyGen})
       \begin{align*} 	\eta_{Y_\pi}: (\C^2,0) &\longrightarrow (Y_\pi,0) \\
                                 (u,v) & \longmapsto  \begin{pmatrix} a_{11} & \cdots & a_{16} \\ \vdots & \vdots & \vdots \\
                            a_{41} & \cdots & a_{46} \end{pmatrix}  \begin{pmatrix} u^4 \\  u^6 \\ u^7 \\ v^6 \\ v^9 \\ v^{11} \end{pmatrix}
       \end{align*}
       This gives us a family of surfaces $(Y_\pi,0) \subset (\C^4,0)$ that are not isomorphic to $(X,0)$ whose Lipschitz saturation $(Y_{\pi}^s,0)$ is toric and isomorphic
       to $(X^s,0)$. 
  \end{exam}

      We will  now a consider a family of hypersurfaces $(X,0) \subset (\C^3,0)$ with equation of the form 
      \[y^N-x^{\alpha N}z^\beta=0,\]
      where $\alpha,\beta \geq 1$ and $\text{ mcd}(\beta, N)=1$. It is a family of toric surface singularities with semigroup $\Ga$ generated
      by $\A=\{(1,0),(\alpha,\beta),(0,N) \}$.
      
  \begin{teo}\label{CalculoLip}
      Let $(X,0)\subset (\C^3,0)$ be the toric hypersurface singularity with normalization map given by 
      \begin{align*}
          \eta:(\C^2,0) &\longrightarrow (X,0) \\
                  (u,v) &\mapsto \left( u, u^\alpha v^\beta, v^N\right),
      \end{align*}
      where $\alpha,\beta \geq 1$ and $\text{ mcd}(\beta, N)=1$. Let $T \subset \N$ be the numerical semigroup generated by $\{N,\beta\}$, and $T^s \subset \N$ be its saturation as in section \ref{secLips}. The monomial $u^av^b \in \C\{u,v\}$ is in the Lipschitz saturation $\Oxo^s$ if and only if $b=mN$  for some $m \in \N$ or  $a \geq \alpha$ and $ b \in T^s$. 
       \end{teo}
 \begin{proof} 
       In this setting the semigroup $\Ga$ is generated by $\A=\{(1,0),(\alpha,\beta),(0,N) \}$, and by the proof of theorem  \ref{SatTor} 
        $u^av^b \in \Oxo^s$ is equivalent to $(a,b) \in \Ga^s$.  \\
        
We first show that if $b=mN$ for some $m \in \N$ or $a \geq \alpha$ and $ b \in T^s$ then $(a,b)\in\Ga^s$. Suppose first that $b=mN$. Since $(1,0),(0,N)\in\Ga$, it follows that $(a,b)\in\Ga\subset \Ga^s$ for all $a\in\N$.

Now suppose that  $a \geq \alpha$ and $ b \in T^s$. Since $\Ga \subset \Ga^s$ we have that
       $u \in \Oxo^s$ and so it is enough to prove the statement for $a=\alpha$. By definition $u^\alpha v^b \in \Oxo^s$ if and only if
       $x_1^\alpha x_2^b -y_1^\alpha y_2^b \in \overline{I_\Delta}$ where 
       \[ I_\Delta=\left< x_1-y_1, x_1^\alpha x_2^\beta-y_1^\alpha y_2^\beta,x_2^N-y_2^N\right> \C\left\{ x_1,x_2,y_1,y_2\right\}. \]  
        The assumption $b \in T^s$ means that $v^b \in \Oo_{C,0}^s \subset \C\{v\}$, which can be rephrased in terms of integral closure of ideals by
        \[ x_2^b-y_2^b \in \overline{\left<x_2^\beta-y_2^\beta, x_2^N-y_2^N \right>}\C\{x_2,y_2\}.\]
        And so if we denote $x_2^b-y_2^b$ by  $f(x_2,y_2)$ we have an integral dependence equation in $\C\{x_2,y_2\}$ of the form:
        \[f^m+g_1(x_2,y_2)f^{m-1} + \cdots + g_m(x_2,y_2) =0,\]
        with $g_k(x_2,y_2) \in J^k=\left< (x_2^N-y_2^N)^i(x_2^\beta-y_2^\beta)^j\,|\, i+j=k\right>\C\{x_2,y_2\}$. Each $g_k$ will then be of the form
        \[g_k(x_2,y_2)= \sum_{j=0}^k h_j(x_2,y_2)(x_2^N-y_2^N)^{k-j}(x_2^\beta-y_2^\beta)^j.\]
        Multiplying the integral dependence equation by $x_1^{\alpha m}$ we obtain:
        \begin{equation}\label{DepInt}
         (x_1^\alpha f)^m + x_1^\alpha g_1\left(x_1^\alpha f \right)^{m-1} + \cdots + x_1^{\alpha(m-1)}g_{m-1}
         x_1^\alpha f + x_1^{\alpha m}g_m=0.\end{equation}
         But now 
         \[ x_1^{\alpha k}g_k(x_2,y_2)= \sum_{j=0}^k x_1^{\alpha (k-j)}h_j(x_2,y_2)(x_2^N-y_2^N)^{k-j}\left[ x_1^\alpha(x_2^\beta-y_2^\beta)\right]^j.\]
         Note that $x_1-y_1, x_1^\alpha x_2^\beta-y_1^\alpha y_2^\beta \in I_\Delta$ implies that $x_1^\alpha(x_2^\beta-y_2^\beta) \in I_\Delta$. In 
         particular we get 
         \[ (x_2^N-y_2^N)^{k-j}\left[ x_1^\alpha(x_2^\beta-y_2^\beta)\right]^j \in I_\Delta^k,\] 
         and so $x_1^{\alpha k}g_k \in I_\Delta^k$. This implies that equation (\ref{DepInt}) is an integral dependence equation for 
         $x_1^\alpha f$ over $I_\Delta$. Finally $x_1^\alpha f=x_1^\alpha \left(x_2^b -y_2^b\right) \in \overline{I_\Delta}$ and $x_1-y_1 \in I_\Delta$
         imply $x_1^\alpha x_2^b -y_1^\alpha y_2^b \in \overline{I_\Delta}$ which is what we wanted to prove. \\

Now we prove that $(a,b)\in\Ga^s$ implies $b=mN$ for some $m \in \N$ or $a \geq \alpha$ and $ b \in T^s$. 

If $b=mN$ for some $m\in\N$, we are done. Assume that $b\neq mN$ for all $m\in\N$. We show that $a\geq\alpha$ and $b\in T^s$.

        Note that for any $u_0 \neq 0$ we have an embedding of the plane toric curve $(C,0)$, with semigroup $T$ generated by $\{N,\beta\}$
        and defined by $y^N-z^\beta =0$, via the map
        \begin{align*}
            (C,0) &\hookrightarrow \left(X, (u_0,0,0)\right) \\ 
            (y,z)  &\mapsto  (u_0,u_0^\alpha y, z) \\
            (v^\beta, v^N) &\mapsto (u_0, u_0^\alpha v^\beta, v^N). 
        \end{align*}
        
        By remark \ref{Rem1} a monomial $u^av^b \in \Oxo^s$ defines a locally Lipschitz meromorphic function on a small enough neighborhood
        $U$ of $0$ in $X$. For any small enough $u_0$ it restricts to a locally Lipschitz meromorphic function $u_0^av^b$ on a neighborhood
        of $(u_0,0,0)$ in the embedded curve $C$. In particular $v^b \in \Oo_{C,0}^s$, or equivalently $b \in T^s$. 

To show that $a\geq\alpha$ we prove that $0\leq a<\alpha$ and $b\neq mN$ implies $(a,b)\notin\Ga^s$. We consider two cases: $a=0$ and $0<a<\alpha$.

             Following example \ref{WU}, note that all points 
       of the form $(0,k)$ with $k<N$ are in $\N^2 \setminus K_+(\Ga)$ and so they are not in $\Ga^s$ by proposition \ref{NoAgrego}. For the points
         of the form $(0,k)$ with $k>N$, $k\neq mN$, consider the ideal 
        \[ I_\Delta=\left< x_1-y_1, x_1^\alpha x_2^\beta-y_1^\alpha y_2^\beta,x_2^N-y_2^N\right> \C\left\{ x_1,x_2,y_1,y_2\right\}.\]  
         Let $\theta \in \C$ be a primitive $N$-th root of unity and consider the arc 
         \begin{align*}
         \varphi:(\C,0) &\to \left(\C^2 \times \C^2, 0\right)\\ t &\mapsto (t^{k+1},t,t^{k+2},\theta t). \end{align*}
         We have the corresponding morphism of analytic algebras $\varphi^*: \C\left\{ x_1,x_2,y_1,y_2\right\} \to \C\{t\}$ 
         such that 
         \begin{align*}
         \left<\varphi^*(I_\Delta)\right>\C\{t\}&=\left<\varphi^*(x_1-y_1), \varphi^*(x_1^\alpha x_2^\beta-y_1^\alpha y_2^\beta),\varphi^*(x_2^N-y_2^N) \right>  \\
               &= \left< (1-t)t^{k+1}, (1-\theta^\beta t^\alpha) t^{\alpha(k+1) + \beta}, 0 \right> \\
               &= \left<t^{k+1}\right>\C\{t\}.\end{align*}
         In particular, 
         \[ \varphi^*(x_2^k -y_2^k)=(1-\theta^k)t^k \notin \left<\varphi^*(I_\Delta)\right>=\left<t^{k+1}\right>.\]
         By \cite[Thm 2.1]{LT08} this impliest that $x_2^k -y_2^k \notin \overline{I_\Delta}$, i.e. $v^k \notin \Oxo^s$. 

All that is left to prove is that for $0<a< \alpha$ and $b \neq mN$ the monomial $u^av^b$ is not in $\Oxo^s$. We will once again use the arc criterion for integral dependence, but this time with the arc 
         \begin{align*}
         \psi:(\C,0) &\to \left(\C^2 \times \C^2, 0\right)\\ t &\mapsto (t^{b+1},t,t^{b+1}+t^r,\theta t), \end{align*}
         where $\theta \in \C$ is a primitive $N$-th root of unity. In this setting the image of $I_\Delta$ by the morphism $\psi^*$ in 
         $\C\{t\}$ is of the form:
          \begin{align*}
         \left<\varphi^*(I_\Delta)\right>\C\{t\}&=\left<\psi^*(x_1-y_1), \psi^*(x_1^\alpha x_2^\beta-y_1^\alpha y_2^\beta),\psi^*(x_2^N-y_2^N) \right>  \\
               &= \left< -t^r, \left[1-\theta^\beta \left( 1 +t^{r-b-1}\right)^\alpha \right] t^{\alpha(b+1) + \beta}, 0 \right> \\
               &= \left<t^{\alpha(b+1) + \beta}\right>\C\{t\} \hspace{0.1in} \text { for }r\text{ big enough.} \end{align*}
         On the other hand we have 
         \[ \psi^*\left(x_1^ax_2^b -y_1^ay_2^b\right)=\left[1-\theta^b\left( 1 +t^{r-b-1}\right)^a \right] t^{a(b+1) + b}. \]
          Using that  $0< a < \alpha$ it is straightforward to verify that   for every $b \geq 0$
          \[ \alpha(b+1) + \beta > a(b+1) +b\]
          and so $x_1^ax_2^b -y_1^ay_2^b  \notin \overline{I_\Delta}$, i.e. $u^av^b \notin \Oxo^s$.
         
         
 \end{proof}
 
\begin{coro}\label{edim-mult hypersurf}
Let $(X,0)\subset (\C^3,0)$ be a germ of toric hypersurface singularity as in theorem \ref{CalculoLip}. Then $(X^s,0)$ is a toric surface singularity of multiplicity $N$ and embedding dimension $N+1$.
 \end{coro}
\begin{proof}
By the combinatorial description of multiplicity in toric geometry, it follows that $N=\mult(X,0)=\mult(X^s,0)$. Similarly, the embedding dimension of $(X^s,0)$ corresponds to the cardinality of the minimal set of generators of $\Ga^s$. We exhibit this minimal set of generators and show that it has cardinality $N+1$.

Theorem \ref{CalculoLip} states that $(a,b)\in\Ga^s$ if and only if $b=mN$ for some $m \in \N$ or $a \geq \alpha$ and $b \in T^s$. We divide the proof in two cases.
\\

\noindent Case $N<\beta$. Write $\beta=kN+l$, $k\geq1$ and $0<l<N$. By the algorithm for computing $T^s$ we obtain $T^s=\{0,N,2N,\ldots,kN,\beta,\beta+1,\beta+2,\cdots\}$ (see section \ref{secLips}). Consider the sets 
\begin{align}
\A&=\{(1,0),(\alpha,\beta),(0,N) \},\notag\\
\B&=\{(\alpha,\beta+1),\ldots,(\alpha,\beta+N-1)\}\setminus\{(\alpha,(k+1)N)\}.\notag
\end{align}
We claim that $\Ga^s=\N(\A\cup\B)$. Firstly observe that $\A\cup\B\subset\Ga^s$. Now let $(a,b)\in\Ga^s$. If $b=mN$ then $(a,b)\in\Ga=\N(\A)\subset\N(\A\cup\B)$. In particular, this holds for the element $(\alpha,(k+1)N)$. Now assume that $a\geq\alpha$ and $b\in T^s$. Since $(1,0)\in\A$ it is enough to consider $a=\alpha$.
\begin{itemize}
\item Suppose $b<\beta$. Then $b=mN$ for some $m\in\N$. Hence $(\alpha,b)\in\N(\A\cup\B)$.
\item Suppose $b=\beta$. Then $(\alpha,b)\in\A\subset\N(\A\cup\B)$.
\item Suppose $\beta+1 \leq b\leq \beta+N-1$, $b\neq (\alpha,(k+1)N)$. Then $(\alpha,b)\in\B\subset\N(\A\cup\B)$.
\item Suppose $\beta+N\leq b$. Write $b-\beta=rN+s$, $r\geq1$, $0\leq s<N$. Then $(\alpha,b)=(\alpha,\beta+rN+s)=(\alpha,\beta+s)+r(0,N)\in\N(\A\cup\B)$.
\end{itemize}
Now we show that $\A\cup\B$ is minimal as a generating set. Indeed, first notice that no element of $\A$ can be generated by $\B$ and viceversa (because of the second entry of the vectors). Similarly, for each $(\alpha,\beta+i)\in\B$ it follows that $(\alpha,\beta+i)\notin\N(\A\cup\B\setminus\{(\alpha,\beta+i)\})$ (because of the first entry of the vectors). Hence $\A\cup\B$ is the minimal generating set of $\Ga^s$ and has cardinality $N+1$.
\\

\noindent Case $\beta<N$. Write $N=k\beta+l$, $k\geq1$ and $0<l<\beta$. We compute $T^s$ as before: $T^s=\{0,\beta,\ldots,k\beta,N,N+1,\cdots\}$. Consider the sets 
\begin{align}
\A=&\{(1,0),(\alpha,\beta),(0,N) \},\notag\\
\B=&\{(\alpha,2\beta),(\alpha,3\beta),\ldots,(\alpha,k\beta)\}\notag\\
&\cup\big[\{(\alpha,N+1),\ldots,(\alpha,N+(N-1))\}\setminus\{(\alpha,N+\beta),\ldots,(\alpha,N+k\beta)\}\big].\notag
\end{align}
We claim that $\Ga^s=\N(\A\cup\B)$. Firstly observe that $\A\cup\B\subset\Ga^s$. Now let $(a,b)\in\Ga^s$. If $b=mN$ then $(a,b)\in\Ga=\N(\A)\subset\N(\A\cup\B)$. Now assume that $a\geq\alpha$ and $b\in T^s$. As before, it is enough to consider $a=\alpha$.
\begin{itemize}
\item For each $i\in\{1,\ldots,k\}$ we have $(\alpha,N+i\beta)=(0,N)+(\alpha,i\beta)\in\N(\A\cup\B)$.
\item Suppose $b<N$. Then $b=m\beta$ for some $m\in\N$. Hence $(\alpha,b)\in\N(\A\cup\B)$.
\item Suppose $b=N$. Then $(\alpha,b)\in\N(\A)\subset\N(\A\cup\B)$.
\item Suppose $N+1 \leq b\leq N+(N-1)$, $b\neq(\alpha,N+i\beta)$, $i\in\{1,\ldots,k\}$. Then $(\alpha,b)\in\B\subset\N(\A\cup\B)$.
\item Suppose $2N\leq b$. Write $b-N=rN+s$, $r\geq1$, $0\leq s<N$. Then $(\alpha,b)=(\alpha,N+rN+s)=r(0,N)+(\alpha,N+s)\in\N(\A\cup\B)$.
\end{itemize}
A similar analysis as in the previous case shows that $\A\cup\B$ is the minimal generating set of $\Ga^s$ and has cardinality $N+1$.
\end{proof}
        
 
    
 \begin{exam}
     Consider the toric hypersurface singularity $(X,0)\subset (\C^3,0)$ with normalization map 
       \begin{align*}
          \eta:(\C^2,0) &\longrightarrow (X,0) \\
                  (u,v) &\mapsto \left( u, u^3 v^{11}, v^5 \right).
      \end{align*}
   It is defined by the equation $y^5-x^{15}z^{11}=0$ and so it has multiplicity $5$.  Following the notation of theorem \ref{CalculoLip} 
   we have the semigroup $T \subset \N$ generated by $\{5,11\}$ and its saturation $T^s \subset \N$ with minimal generating set
   $\{5,11,12,13,14\}$. 
   
   \begin{figure}[H]
       \centering
       \includegraphics[width=\textwidth]{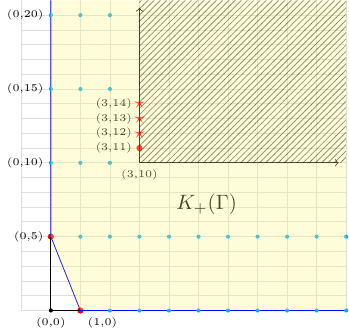}
       \caption{The integral points of the shaded region are contained in $\Ga^s$. The highlighted points mark the minimal generating set of $\Ga^s$.}
       \end{figure}
   
   From the previous corollary we have the semigroup $\Ga^s$ associated to the Lipschitz saturation $(X^s,0)$ with minimal generating 
   set \[\mathcal{A}^s=\{(1,0), (3,11), (3,12), (3,13), (3,14),(0,5)\},\] normalization map 
    \begin{align*}
          \eta:(\C^2,0) &\longrightarrow (X^s,0)\subset (\C^6,0) \\
                  (u,v) &\mapsto \left( u, u^3 v^{11}, u^3 v^{12},u^3 v^{13},u^3 v^{14},v^5 \right),
      \end{align*}
   and embedding dimension $6$.
       \end{exam}
       
    Interchanging the roles of $5$ and $11$ we get the following example.
    
    \begin{exam}
     Consider the toric hypersurface singularity $(X,0)\subset (\C^3,0)$ with normalization map 
       \begin{align*}
          \eta:(\C^2,0) &\longrightarrow (X,0) \\
                  (u,v) &\mapsto \left( u, u^3 v^5, v^{11} \right).
      \end{align*}
   It is defined by the equation $y^{11}-x^{33}z^{5}=0$ and so it has multiplicity $11$.  Following the notation of theorem \ref{CalculoLip} 
   we have again the semigroup $T \subset \N$ generated by $\{5,11\}$ and its saturation $T^s \subset \N$ with minimal generating set
   $\{5,11,12,13,14\}$. 
   
   \begin{figure}[H]
       \centering
       \includegraphics[width=\textwidth]{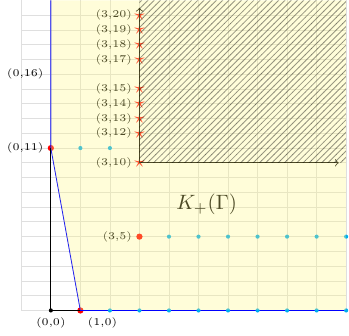}
        \caption{The integral points of the shaded region are contained in $\Ga^s$. The highlighted points mark the minimal generating set of $\Ga^s$.}
       \end{figure}
   
   From the previous corollary we have the semigroup $\Ga^s$ associated to the Lipschitz saturation $(X^s,0)$ with minimal generating 
   set 
\small{
\[\mathcal{A}^s=\{(1,0), (3,5), (3,10), (3,12), (3,13), (3,14), (3,15), (3,17), (3,18) , (3,19) , (3,20), (0,11)\},\] }
normalization map 
    \begin{align*}
          \eta:(\C^2,0) &\longrightarrow (X^s,0)\subset (\C^{12},0) \\
                  (u,v) &\mapsto \left( u, u^3v^5, u^3v^{10}, u^3 v^{12}, u^3 v^{13},u^3 v^{14},u^3 v^{15}, u^3 v^{17},
                  u^3 v^{18},u^3 v^{19},u^3 v^{20},v^{11} \right),
      \end{align*}
   and embedding dimension $12$.
       \end{exam}
       
        Note that in both examples we have $(3,10) + \N^2 \subset \Ga^s$. 
   
\begin{rem}
Recall from section \ref{secLips} that the Lipschitz saturation of an irreducible curve has multiplicity equal to its embedding dimension. The results and examples from this section shows that there is no general relation among the embedding dimension and the multiplicity of the Lipschitz saturation in higher dimensions.
\end{rem}

\begin{rem}
Contrary to the case of curves, for the moment it is not clear that there exists an explicit algorithm to compute the semigroup $\Gamma^s$ from the semigroup $\Gamma$, in general. We proved in theorem \ref{SatTor} that $\Gamma^s$ is finitely generated by Noetherian arguments, hence we do not have a constructive procedure to obtain a set of generators. 
However, the same theorem shows that the toric structure is preserved under Lipschitz saturation so, \textit{in principle}, there should be a semigroup operation that provides $\Gamma^s$. It remains an open question to find an algorithm that produces generators for $\Gamma^s$ in general.
\end{rem}

\section*{Acknowledgements}    

The authors would like to thank professors P. Gonzalez Perez, C. Huneke and I. Swanson for the fruitful e-mail exchanges that greatly helped in the preparation of this work. They also acknowledge support by PAPIIT grant IN117523 and CONAHCYT grant CF-2023-G-33. A. Giles Flores acknowledges support by UAA grants PIM21-1 and PIM24-7.

\vspace{.5cm}
{\footnotesize \textsc {D. Duarte, Centro de Ciencias Matem\'aticas, UNAM.} \\
E-mail: adduarte@matmor.unam.mx}\\
{\footnotesize \textsc {A. Giles Flores, Universidad Aut\'onoma de Aguascalientes.} \\
Email: arturo.giles@edu.uaa.mx}
\end{document}